\documentclass{article}

\usepackage{amsfonts}
\usepackage{amsmath}
\usepackage{amsthm}
\usepackage{enumerate}

\newcommand{\R}{\mathbb{R}}

\newcommand{\M}{\mathcal{M}}
\DeclareMathOperator{\Lip}{Lip}

\DeclareMathOperator{\support}{supp}

\newtheorem{proposition}{Proposition}

\newtheorem{corollary}[proposition]{Corollary}
\newtheorem*{corollary*}{Corollary}
\newtheorem{theorem}[proposition]{Theorem}
\newtheorem*{theorem*}{Theorem}

\theoremstyle{definition}

\theoremstyle{remark}

\title{Families of extensions of the Kantorovich-Rubinstein and Lipschitz norms}

\author{
  D\'avid Terj\'ek \\
  Alfr\'ed R\'enyi Institute of Mathematics\\
  Budapest, Hungary\\
  \texttt{dterjek@renyi.hu} \\
}

\begin{document}

\maketitle

\begin{abstract}
We propose a family of extensions of the Kantorovich-Rubinstein norm from the space of zero-charge countably additive measures on a compact metric space to the space of all countably additive measures, and a family of extensions of the Lipschitz norm from the quotient space of Lipschitz functions on a compact metric space to the space of all Lipschitz functions. These families are parameterized by $p,q \in [1,\infty]$, and if $p,q$ are Hölder conjugates, then the dual of the resulting $p$-Kantorovich space is isometrically isomorphic to the resulting $q$-Lipschitz space.
\end{abstract}

\section{Introduction}

Given a compact metric space $(X,d)$, the vector space $\M(X,0)$ of countably additive measures $\mu$ on the Borel $\sigma$-algebra of $X$ such that $\mu(X)=0$ can be normed by the Kantorovich-Rubinstein norm
\[
\Vert \mu \Vert_{KR} = \inf_{\pi \in \M(X \times X), \pi(X \times \cdot) - \pi(\cdot \times X) = \mu}\left\{ \int d(x,y) d\pi(x,y) \right\},
\]
the theory of which was developed in \cite{Kantorovichetal1957} and \cite{Kantorovichetal1958}. See \cite[Section~VIII.4]{Kantorovichetal1982} or \cite[Section~8.4]{Cobzasetal2019} for summaries. The topological dual space can be identified with the quotient space of Lipschitz functions with respect to constant functions, or equivalently the space $\Lip(X,x_0)$ of Lipschitz functions vanishing at an arbitrary base point $x_0=X$, equipped with the Lipschitz norm
\[
\Vert f \Vert_L = \sup_{x,y \in X, x \neq y}\left\{ \frac{\vert f(x)-f(y) \vert}{d(x,y)} \right\}.
\]

We propose to extend the Kantorovich-Rubinstein norm to the space $\M(X)$ of all countably additive measures on $X$ as
\begin{equation} \label{p_norm}
\Vert \mu \Vert_{pK} = \inf_{\xi \in \M(X,0)}\left\{ (\Vert \xi \Vert_{KR}^p + \Vert \mu - \xi \Vert_{TV}^p)^{\frac{1}{p}} \right\}
\end{equation}
with $p \in [1,\infty]$ and $\Vert.\Vert_{TV}$ being the total variation norm, and to extend the Lipschitz norm to the space $\Lip(X)$ of all Lipschitz functions on $X$ as
\begin{equation} \label{q_norm}
\Vert f \Vert_{qL} = (\Vert f \Vert_L^q + \Vert f \Vert_\infty^q)^{\frac{1}{q}}
\end{equation}
with $q \in [1,\infty]$ and $\Vert.\Vert_\infty$ being the sup norm. The limiting cases $p,q=\infty$ are interpreted as usual. In the following, we will prove some properties of the resulting $p$-Kantorovich $(\M(X),\Vert.\Vert_{pK})$ and $q$-Lipschitz $(Lip(X), \Vert . \Vert_{qL})$ spaces, including that if $p,q$ are Hölder conjugates, i.e. $\frac{1}{p}+\frac{1}{q}=1$, then the topological dual of the $p$-Kantorovich space can be identified with the $q$-Lipschitz space.

The theory of the $p=1$ and $q=\infty$ case was developed by L. G. Hanin in a series of papers \cite{Hanin1992, Hanin1994, Hanin1997, Hanin1999}. See \cite[Section~8.5]{Cobzasetal2019} for a summary. The case $p=\infty$ and $q=1$ was proposed in dual form in \cite{Chitescuetal2014}. Equivalence to the primal form was shown in \cite{Terjek2021}. For $\mu,\nu$ probability measures, $\Vert \mu - \nu \Vert_{\infty K}$ is also known as the Fortet-Mourier distance.

\section{$p$-Kantorovich and $q$-Lipschitz norms}


The following propositions show that \eqref{p_norm} and \eqref{q_norm} define families of equivalent norms, with the spaces $(\M(X),\Vert.\Vert_{pK})$ and $(Lip(X), \Vert . \Vert_{qL})$ being incomplete and complete, respectively, while the pointwise product turns $q$-Lipschitz spaces into Gelfand algebras.

\begin{proposition} \label{prop_separating}
Given Hölder conjugates $p,q \in [1,\infty]$, for any $\mu \in \M(X)$ and $f \in \Lip(X)$, one has
\begin{equation} \label{separating}
\left\vert \int f d\mu \right\vert \leq \Vert \mu \Vert_{pK} \Vert f \Vert_{qL}.
\end{equation}
\end{proposition}
\begin{proof}
For any $f \in \Lip(X)$, $\mu \in \M(X)$ and $\xi \in \M(X,0)$, one has
\[
\left\vert \int f d\mu \right\vert 
\leq \left\vert \int f d\xi \right\vert + \left\vert \int f d(\mu-\xi) \right\vert
\leq \Vert f \Vert_L \Vert \xi \Vert_{KR} + \Vert f \Vert_\infty \Vert \mu-\xi \Vert_{TV}
\]
\[
\leq (\Vert f \Vert_L^q + \Vert f \Vert_\infty^q)^{\frac{1}{q}} (\Vert \xi \Vert_{KR}^p + \Vert \mu-\xi \Vert_{TV}^p)^{\frac{1}{p}}
= \Vert f \Vert_{qL} (\Vert \xi \Vert_{KR}^p + \Vert \mu-\xi \Vert_{TV}^p)^{\frac{1}{p}}
\]
where the Hölder inequality for $\R^2$ was used. Since this holds for any $\xi$, the proposition follows from the definition of $\Vert.\Vert_{pK}$.
\end{proof}

\begin{proposition}
Given $p,q \in [1,\infty]$, $\Vert . \Vert_{pK}$ is a norm on $\M(X)$ and $\Vert . \Vert_{qL}$ is a norm on $\Lip(X)$.
\end{proposition}
\begin{proof}
Both functionals are clearly seminorms. By Proposition~\ref{prop_separating} they are separating, so that they are norms as well.
\end{proof}

\begin{proposition} \label{prop_equivalent_norms}
Given $p_1,p_2 \in [1,\infty]$, the norms $\Vert.\Vert_{p_1K}$ and $\Vert.\Vert_{p_2K}$ on $\M(X)$ are equivalent. Given $q_1,q_2 \in [1,\infty]$, the norms $\Vert.\Vert_{q_1L}$ and $\Vert.\Vert_{q_2L}$ on $\Lip(X)$ are equivalent.
\end{proposition}
\begin{proof}
Suppose that $p_1 \leq p_2$ and $q_1 \leq q_2$. One has
\[
2^{-\frac{1}{p_1}} (\Vert \xi \Vert_{KR}^{p_1} + \Vert \mu - \xi \Vert_{TV}^{p_1})^{\frac{1}{p_1}}
\leq \max\{ \Vert \xi \Vert_{KR}, \Vert \mu - \xi \Vert_{TV} \}
\leq (\Vert \xi \Vert_{KR}^{p_2} + \Vert \mu - \xi \Vert_{TV}^{p_2})^{\frac{1}{p_2}}
\]
\[
\leq 2^{\frac{1}{p_2}} \max\{ \Vert \xi \Vert_{KR}, \Vert \mu - \xi \Vert_{TV} \}
\leq 2^{\frac{1}{p_2}} (\Vert \xi \Vert_{KR}^{p_1} + \Vert \mu - \xi \Vert_{TV}^{p_1})^{\frac{1}{p_1}}
\]
for any $\mu \in \M(X)$ and $\xi \in \M(X,0)$. Similarly, one has
\[
2^{-\frac{1}{q_1}} (\Vert f \Vert_L^{q_1} + \Vert f \Vert_\infty^{q_1})^{\frac{1}{q_1}}
\leq \max\{ \Vert f \Vert_L, \Vert f \Vert_\infty \}
\leq (\Vert f \Vert_L^{q_2} + \Vert f \Vert_\infty^{q_2})^{\frac{1}{q_2}}
\]
\[
\leq 2^{\frac{1}{q_2}} \max\{ \Vert f \Vert_L, \Vert f \Vert_\infty \}
\leq 2^{\frac{1}{q_2}} (\Vert f \Vert_L^{q_1} + \Vert f \Vert_\infty^{q_1})^{\frac{1}{q_1}}
\]
for any $f \in \Lip(X)$.
\end{proof}

\begin{proposition} \label{prop_completeness}
Given $p,q \in [1,\infty]$, the normed space $(\M(X),\Vert.\Vert_{pK})$ is not complete if $X$ is an infinite set, while the normed space $(Lip(X), \Vert . \Vert_{qL})$ is complete for any $X$, i.e. a Banach space.
\end{proposition}
\begin{proof}
The following is an adaptation of the proof of \cite[Theorem~8.4.7]{Cobzasetal2019}. If $X$ is infinite, it has an accumulation point $x$. Let $(x_n)$ be a sequence in $X \setminus \{x\}$ converging to $x$. Since $\Vert \delta_x - \delta_{x_n} \Vert_{pK} \leq d(x,x_n)$, one has $\lim_{n \to \infty} \Vert \delta_x - \delta_{x_n} \Vert_{pK} = 0$. On the other hand, $\Vert \delta_x - \delta_{x_n} \Vert_{TV} = 2$, so that $\delta_{x_n}$ does not converge to $\delta_x$ with respect to the topology induced by $\Vert.\Vert_{TV}$, meaning that $\Vert.\Vert_{TV}$ and $\Vert.\Vert_{pK}$ are not equivalent norms. If $(\M(X),\Vert.\Vert_{pK})$ was complete, the identity operator from $(\M(X),\Vert.\Vert_{pK})$ to $(\M(X),\Vert.\Vert_{TV})$ would be an isomoprhism by the Banach isomorphism theorem, which would imply the equivalence of the norms $\Vert.\Vert_{TV}$ and $\Vert.\Vert_{pK}$.

The following is an adaptation of the proof of \cite[Theorem~8.1.3]{Cobzasetal2019}. Suppose that $(f_n)$ is a Cauchy sequence with respect to $\Vert . \Vert_{qL}$. Then it is also Cauchy with respect to $\Vert . \Vert_\infty$, hence converges uniformly to some bounded function $f : X \to \R$. Since it is also Cauchy with respect to $\Vert . \Vert_L$, by \cite[Lemma~8.1.4]{Cobzasetal2019} it also converges to the same $f$ with respect to $\Vert . \Vert_L$, and $f$ is Lipschitz. Consequently, $(f_n)$ converges to $f$ with respect to both $\Vert . \Vert_\infty$ and $\Vert . \Vert_L$, hence with respect to $\Vert . \Vert_{qL}$ as well, so that the space $(Lip(X), \Vert . \Vert_{qL})$ is complete.
\end{proof}

\begin{corollary}
Given $q \in [1,\infty]$, defining the product of $f,g \in Lip(X)$ as $fg(x)=f(x)g(x)$ turns $(Lip(X), \Vert . \Vert_{qL})$ into a complete normed algebra whose product is continuous, i.e. a Gelfand algebra.
\end{corollary}
\begin{proof}
The result is known for $q=1$ and $q=\infty$, see \cite[Section~7.1]{Weaver2018}. The result follows by the equivalence of $q$-Lipschitz norms by Proposition~\ref{prop_equivalent_norms} and the completeness of $(Lip(X), \Vert . \Vert_{qL})$ by Proposition~\ref{prop_completeness}.
\end{proof}

\section{Duality of $p$-Kantorovich and $q$-Lipschitz spaces}

The following proposition is needed to prove duality.

\begin{proposition} \label{prop_dense}
Given $p \in [1,\infty]$, the set of all measures with finite support is dense in $(\M(X),\Vert.\Vert_{pK})$.
\end{proposition}
\begin{proof}
By \cite[Lemma~2]{Hanin1999}, the set of all measures with finite support is dense in $(\M(X),\Vert.\Vert_{1K})$. Since the norms $\Vert.\Vert_{1K}$ and $\Vert.\Vert_{pK}$ are equivalent for any $p \in [1,\infty]$ by Proposition~\ref{prop_equivalent_norms}, they generate the same topology, implying the proposition.
\end{proof}

We are going to apply techniques of convex analysis to show that the dual of the $p$-Kantorovich space can be identified with the $q$-Lipschitz space if $\frac{1}{p} + \frac{1}{q} = 1$.

\begin{theorem} \label{prop_p_q_duality}
Given Hölder conjugates $p,q \in [1,\infty]$, for any $f \in Lip(X)$ the functional $u_f : (\mathcal{M}(X),\Vert.\Vert_{pK}) \to \R$ defined by $u_f(\mu) = \int f d\mu$ is linear and continuous with $\Vert u_f \Vert_{pK}^* = \Vert f \Vert_{qL}$. Moreover, every continuous linear functional $v$ on $(\mathcal{M}(X), \Vert . \Vert_{pK})$ is of the form  $v(\mu)=u_f(\mu)$ for a uniquely determined function $f \in Lip(X)$ with $\Vert v \Vert_{pK}^* = \Vert f \Vert_{qL}$. Consequently, the mapping $(f \to u_f)$ is an isometric isomorphism of $(Lip(X), \Vert . \Vert_{qL})$ onto the topological dual $(\mathcal{M}(X), \Vert . \Vert_{pK})^*$, i.e.
\begin{equation}
(Lip(X), \Vert . \Vert_{qL}) \cong (\mathcal{M}(X), \Vert . \Vert_{pK})^*.
\end{equation}
\end{theorem}
\begin{proof}
It follows from \eqref{separating} that $u_f$ is a bounded and linear functional on $(\mathcal{M}(X), \Vert . \Vert_{pK})$.

Consider the duality of $(\mathcal{M}(X), \Vert . \Vert_{1K})$ and $(\Lip(X),\Vert.\Vert_{\infty L})$ given by \cite[Theorem~1]{Hanin1999}. For $p \in [1,\infty]$, let the indicators $\iota_p : (\mathcal{M}(X), \Vert . \Vert_{1K}) \to \overline{\R}$ be defined as
\[
\iota_p(\mu) = \begin{cases}
0 \text{ if } \Vert \mu \Vert_{pK} \leq 1 \\
\infty \text{ otherwise }
\end{cases}.
\]
Their convex conjugates $\iota_p^* : (\Lip(X),\Vert.\Vert_{\infty L}) \to \overline{\R}$ are defined as
\[
\iota_p^*(f) = \sup_{\mu \in \M(X)}\left\{ \int f d\mu - \iota_p(\mu) \right\},
\]
so that $\iota_p^*(f) = \sup_{\Vert \mu \Vert_{pK} \leq 1}\left\{ \int f d\mu \right\} = \Vert u_f \Vert_{pK}^*$ is exactly the dual norm of $u_f$. We claim that $\iota_p^*(f) = \Vert f \Vert_{qL}$, which would prove that the linear map $(f \to u_f)$ is an isometry.

Let $H : (\M(X,0),\Vert.\Vert_{KR}) \times (\M(X),\Vert.\Vert_{1K}) \to \R^2$ be defined as 
\[
H(\xi_1,\xi_2) = (\Vert \xi_1 \Vert_{KR}, \Vert \xi_2 \Vert_{TV}).
\]
Let $G : \R^2 \to \overline{\R}$ be defined as 
\[
G(x) = \begin{cases}
0 \text{ if } \Vert x \Vert_p \leq 1,\\
\infty \text{ otherwise },
\end{cases}
\]
i.e. $G$ is the indicator of the unit ball of the $l^p$ norm on $\R^2$, so that its conjugate is $G^*(y) = \Vert y \Vert_q$, i.e. the $l^q$ norm. With the usual ordering on $\R^2$, $H$ is clearly convex while $g$ is proper, convex and increasing. Then the mapping $\varphi = g \circ H$ is convex. We are going to invoke \cite[Theorem~2.8.10]{Zalinescu2002}. In the notation of the theorem, $D = Y_0 = \R^2$ in our case, so that condition (vi) of the theorem clearly holds. This implies that the conjugate $\varphi^* : (\Lip(X,x_0),\Vert.\Vert_L) \times (\Lip(X),\Vert.\Vert_{\infty L}) \to \overline{\R}$ is
\[
\varphi^*(f_1,f_2) = \min_{y \in \R_+^2}\{ (y_1 \Vert . \Vert_{KR} + y_2 \Vert . \Vert_{TV})^*(f_1,f_2) + \Vert y \Vert_q \}.
\]

By \cite[Theorem~2.3.1(v)]{Zalinescu2002}, \cite[Theorem~2.3.1(viii)]{Zalinescu2002} and the well known conjugate relations 
\[
(\xi_1 \to \Vert \xi_1 \Vert_{KR})^* = \left(f_1 \to \begin{cases}
0 \text{ if } \Vert f_1 \Vert_L \leq 1,\\
\infty \text{ otherwise}
\end{cases}
\right)
\]
and
\[
(\xi_2 \to \Vert \xi_2 \Vert_{TV})^* = \left(f_2 \to \begin{cases}
0 \text{ if } \Vert f_2 \Vert_\infty \leq 1,\\
\infty \text{ otherwise}
\end{cases}
\right),
\]
the conjugate of the mapping
\[
(x_1,x_2) \to y_1 \Vert \xi_1 \Vert_{KR} + y_2 \Vert \xi_2 \Vert_{TV}
\]
is the mapping
\[
(f_1,f_2) \to
\begin{cases}
0 \text{ if } \Vert f_1 \Vert_L \leq y_1 \text{ and } \Vert f_2 \Vert_\infty \leq y_2,\\
\infty \text{ otherwise,}
\end{cases}
\]
so that one has
\[
\varphi^*(f_1,f_2) = \min_{y \in \R_+^2, \Vert f_1 \Vert_L \leq y_1, \Vert f_2 \Vert_\infty \leq y_2}\{ \Vert y \Vert_q \} = (\Vert f_1 \Vert_L^q+\Vert f_2 \Vert_\infty^q)^{\frac{1}{q}}.
\]

Consider the linear map $A \in L((\M(X,0),\Vert.\Vert_{KR}) \times (\M(X),\Vert.\Vert_{1K}), (\M(X),\Vert.\Vert_{1K}))$ defined as
\[
A(\xi_1,\xi_2) = \xi_1 + \xi_2,
\]
which is clearly bounded, hence continuous.

Its adjoint $A^* \in L((\Lip(X),\Vert.\Vert_{\infty L}), (\Lip(X,x_0),\Vert.\Vert_L) \times (\Lip(X),\Vert.\Vert_{\infty L}))$ is given by
\[
A^* f = (f-f(x_0), f).
\]

By \cite[Theorem~2.3.1(ix)]{Zalinescu2002}, the conjugate of
\[
A\varphi(\mu) = \inf_{(\xi_1,\xi_2) \in (\M(X,0),\Vert.\Vert_{KR}) \times (\M(X),\Vert.\Vert_{1K}), \mu = A(\xi_1,\xi_2)}\{ \varphi(\xi_1,\xi_2) \} = \iota_p(\mu)
\]
is the mapping
\[
\varphi^* \circ A^*(f) = (\Vert f-f(x_0) \Vert_L^q+\Vert f \Vert_\infty^q)^{\frac{1}{q}} = \Vert f \Vert_{qL},
\]
proving the claim, so that $(f \to u_f)$ is an isometry.

To see that $(f \to u_f)$ is onto $\Lip(X)$, take any $v \in (\M(X),\Vert.\Vert_{pK})^*$ and set $f(x) = v(\delta_x)$ for all $x \in X$. One has $\vert f(x) \vert \leq \Vert v \Vert_{pK}^* \Vert \delta_x \Vert_{pK} = \Vert v \Vert_{pK}^*$ and $\vert f(x) - f(y) \vert \leq \Vert v \Vert_{pK}^* \Vert \delta_x - \delta_y \Vert_{pK} \leq \Vert v \Vert_{pK}^* d(x,y)$ for any $x,y \in X$, so that $f \in \Lip(X)$. One has $u_f(\delta_x) = v(\delta_x)$, so that by linearity, $u_f(\mu) = v(\mu)$ for any $\mu \in \M(X)$ with finite support. Since such measures are dense in $(\M(X),\Vert.\Vert_{pK})$ by Proposition~\ref{prop_dense}, one has $u_f=v$, completing the proof.
\end{proof}

We get the following dual representation of $\Vert . \Vert_{pK}$ as a corollary.

\begin{corollary} \label{prop_dual_representation}
Given $\mu \in \M(X)$, one has the dual representation
\begin{equation}
\Vert \mu \Vert_{pK} = \sup_{f \in Lip(X), \Vert f \Vert_{qL} \leq 1}\left\{ \int f d\mu \right\},
\end{equation}
and there exists $f_* \in Lip(X)$ such that $\Vert f_* \Vert_{qL} = 1$ and $\int f_* d\mu = \Vert \mu \Vert_{pK}$.
\end{corollary}
\begin{proof}
This follows from Proposition~\ref{prop_p_q_duality} and the Hahn-Banach theorem.
\end{proof}

\section{Optimality conditions}

The following proposition shows that the infimum in \eqref{p_norm} is always attained.

\begin{proposition} \label{prop_optimal_xi}
For any $\mu \in \M(X)$, there exists $\xi_* \in \M(X,0)$ such that $\Vert \mu \Vert_{pK} = (\Vert \xi_* \Vert_{KR}^p + \Vert \mu - \xi_* \Vert_{TV}^p)^{\frac{1}{p}}$.
\end{proposition}
\begin{proof}
The following is an adaptation of the proof of \cite[Proposition~6]{Hanin1999}. Let $\xi_n$ be a sequence in $\M(X,0)$ such that $\Vert \mu \Vert_{pK} = \lim_{n\to\infty}(\Vert \xi_n \Vert_{KR}^p + \Vert \mu - \xi_n \Vert_{TV}^p)^{\frac{1}{p}}$. Without loss of generality one can assume that $\Vert \xi_n \Vert_{TV} \leq R$ for all $n$ for some constant $R > 0$. By the weak-* compactness of the $\Vert.\Vert_{TV}$-ball of radius $R$ by \cite[Theorem~8.4.25]{Cobzasetal2019}, up to a subsequence, $\xi_n$ weak-* converges to some $\xi_* \in \M(X,0)$. Let $a = \liminf_{n \to \infty} \Vert \xi_n \Vert_{KR}$ and $b = \liminf_{n \to \infty} \Vert \mu - \xi_n \Vert_{TV}$, so that $\Vert \mu \Vert_{pK} \geq (a^p+b^p)^{\frac{1}{p}}$. For every $n$ and any $f \in \Lip(X)$ with $\Vert f \Vert_L \leq 1$ one has $\left\vert \int f d\xi_n \right\vert \leq \Vert \xi_n \Vert_{KR}$, so that $\left\vert \int f d\xi_* \right\vert \leq a$, hence $\Vert \xi_* \Vert_{KR} \leq a$. Similarly, for every $n$ and any $f \in \Lip(X)$ with $\Vert f \Vert_\infty \leq 1$ one has $\left\vert \int f d(\mu-\xi_n) \right\vert \leq \Vert \mu-\xi_n \Vert_{TV}$, so that $\left\vert \int f d(\mu-\xi_*) \right\vert \leq b$, hence $\Vert \mu-\xi_* \Vert_{TV} \leq b$. Thus $(\Vert \xi_* \Vert_{KR}^p + \Vert \mu - \xi_* \Vert_{TV}^p)^{\frac{1}{p}} \leq (a^p+b^p)^{\frac{1}{p}} \leq \Vert \mu \Vert_p$, implying the proposition.
\end{proof}

It is well known that for any $\xi \in \M(X,0)$ there exists $\pi_* \in \M(X \times X)$ such that $\pi_*(X \times \cdot) - \pi_*(\cdot \times X) = \xi$ and $\Vert \xi \Vert_{KR} = \int d(x,y) d\pi_*(x,y)$ (see e.g. \cite[Section~VIII.4]{Kantorovichetal1982}). By this and Proposition~\ref{prop_optimal_xi}, there exists an optimal pair $\xi_* \in \M(X,0)$, $\pi_* \in \M(X \times X)$ such that $\pi_*(X \times \cdot) - \pi_*(\cdot \times X) = \xi_*$ and
\[
\Vert \mu \Vert_{pK} = \left( \left(\int d(x,y) d\pi_*(x,y)\right)^p + \Vert \xi_* \Vert_{TV}^p \right)^{\frac{1}{p}},
\]
while by Proposition~\ref{prop_dual_representation} there exists an optimal $f_* \in Lip(X)$ such that $\Vert \mu \Vert_{pK} = \int f_* d\mu$ and $\Vert f \Vert_{qL} = 1$. The following proposition characterizes such optimal variables, generalizing \cite[Proposition~7]{Hanin1999}.

\begin{proposition}
Given $\mu \in \M(X)$, the measures $\xi_* \in \M(X,0)$, $\pi_* \in \M(X \times X)$ with $\pi_*(X \times \cdot) - \pi_*(\cdot \times X) = \xi_*$ are optimal if and only of there exists a $f_* \in Lip(X)$ such that the conditions
\begin{enumerate}[(i)]
\item $\Vert f_* \Vert_{qL} = 1$,
\item $\Vert f_* \Vert_L \Vert \xi_* \Vert_{KR} + \Vert f_* \Vert_\infty \Vert \mu - \xi_* \Vert_{TV} = (\Vert \xi_* \Vert_{KR}^p + \Vert \mu - \xi_* \Vert_{TV}^p)^{\frac{1}{p}}$,
\item $f_*(x)-f_*(y) = \Vert f_* \Vert_L d(x,y)$ if $(x,y) \in \support(\pi_*)$ and
\item $f_*(x) = \pm \Vert f_* \Vert_\infty$ if $x \in \support((\mu-\xi_*)_\pm)$
\end{enumerate}
are satisfied. In this case, $f_*$ is optimal, i.e. $\int f_* d\mu = \Vert \mu \Vert_{pK}$.
\end{proposition}
\begin{proof}
Let $f_* \in \Lip(X)$ be a function satisfying the above conditions. Then one has
\[
\Vert \mu \Vert_{pK} \geq \int f_* d\mu 
= \int f_* d\xi_* + \int f_* d(\mu-\xi_*) 
\]
\[
= \int f_*(x)-f_*(y) d\pi_*(x,y) + \Vert f_* \Vert_\infty\Vert \mu - \xi_* \Vert_{TV}
\]
\[
= \Vert f_* \Vert_L \int d(x,y) d\pi_*(x,y) + \Vert f_* \Vert_\infty\Vert \mu - \xi_* \Vert_{TV}
\]
\[
\geq \Vert f_* \Vert_L \Vert \xi_* \Vert_{KR} + \Vert f_* \Vert_\infty\Vert \mu - \xi_* \Vert_{TV}
\]
\[
= (\Vert \xi_* \Vert_{KR}^p + \Vert \mu - \xi_* \Vert_{TV}^p)^{\frac{1}{p}}
\geq \Vert \mu \Vert_{pK},
\]
so that the conditions are sufficient.

Now let $xi_*,\pi_*$ and $f_*$ be optimal. Clearly $(i)$ is satisfied. On the other hand, one has
\[
\Vert \mu \Vert_{pK} 
= \int f_* d\mu = \int f_* d\xi_* + \int f_* d(\mu-\xi_*) 
\]
\[
\leq \int f_*(x)-f_*(y) d\pi_*(x,y) + \Vert f_* \Vert_\infty\Vert \mu - \xi_* \Vert_{TV}
\]
\[
\leq \Vert f_* \Vert_L \int d(x,y) d\pi_*(x,y) + \Vert f_* \Vert_\infty\Vert \mu - \xi_* \Vert_{TV}
\]
\[
= \Vert f_* \Vert_L \Vert \xi_* \Vert_{KR} + \Vert f_* \Vert_\infty\Vert \mu - \xi_* \Vert_{TV}
\]
\[
\leq (\Vert f \Vert_L^q + \Vert f \Vert_\infty^q)^{\frac{1}{q}} (\Vert \xi_* \Vert_{KR}^p + \Vert \mu - \xi_* \Vert_{TV}^p)^{\frac{1}{p}}
= \Vert \mu \Vert_{pK},
\]
implying $(ii), (iii)$ and $(iv)$. This shows that the conditions are necessary as well.
\end{proof}

\section*{Acknowledgements}
The author is supported by the Hungarian National Excellence Grant 2018-1.2.1-NKP-00008 and by the Hungarian Ministry of Innovation and Technology NRDI Office within the framework of the Artificial Intelligence National Laboratory Program.



\bibliographystyle{apalike}
\bibliography{pqhkr}

\end{document}